\documentclass[11pt]{amsart}
\usepackage{fullpage}
\usepackage{amscd,amssymb,amsfonts,amsmath,amsthm}
\usepackage{amssymb,amsfonts,amsthm}
\usepackage{amsmath,comment}
\usepackage{enumerate}
\usepackage[all]{xy}
\usepackage[pdftex]{graphicx}
\usepackage[pdftex]{hyperref}

\newtheorem*{thm*}{Theorem}
\newtheorem{thm}{Theorem}[section]
\newtheorem{prop}[thm]{Proposition}

\theoremstyle{definition}

\newtheorem{defn}[thm]{Definition}

\newcommand{\Ab}{{\mathbb A}}
\newcommand{\Cb}{{\mathbb C}}

\newcommand{\Gb}{\mathbb{G}}

\newcommand{\Pb}{{\mathbb P}}

\newcommand{\Qb}{{\mathbb Q}}

\newcommand{\Oc}{{\mathcal O}}

\newcommand{\SL}{\textrm{SL}}

\newcommand{\Sym}{\textrm{Sym}}

\newtheorem{theo}{Theorem}
\newtheorem{lem}[theo]{Lemma}

\theoremstyle{definition}

\newcommand{\Qbar}{\overline{\Qb}}
\newcommand{\OO}{\mathcal{O}}
\newcommand{\PGL}{\text{PGL}}

\title{Positivity of GIT heights of zero-cycles and hyperplane arrangements}
\author{Ashwath Rabindranath}
\author{William F. Sawin}

\address{Department of Mathematics, University of Michigan}
\email{ashwathr@umich.edu}

\address{Department of Mathematics, Princeton University}
\email{wsawin@math.princeton.edu}
\date{}
\begin{document}
\begin{abstract}
In this paper, we give a partial answer to a conjecture of Zhang from 1996 about the GIT height function defined in \cite{Z} for semi-stable algebraic cycles. We prove the conjecture for zero-cycles and hyperplane arrangements.
\end{abstract}
\maketitle
\vspace{-0.2in}
\section{Introduction}
\noindent In the early 1990's, as part of the development of arithmetic intersection theory and  Arakelov theory, the subject of \textit{arithmetic} geometric invariant theory arose in the works \cite{Bu} and \cite{Z} by Burnol and Zhang respectively. In \cite{Z}, a notion of \textit{GIT height function} (denoted $\hat{h}$) is introduced for \textit{Chow semi-stable} algebraic cycles $X \subset \Pb^N$. The goal of this paper is to prove the positivity of this height function when $X$ is either a zero-cycle or a hyperplane arrangement. This is a special case of the more general positivity conjecture of Zhang from \cite{Z}. In particular, we prove the following two theorems.

\begin{thm} If $X$ is a \textit{Chow semi-stable} zero-cycle, $\hat{h}(X) \geq 0$. If $X$ is \textit{Chow stable}, then $\hat{h}(X) > 0$. \end{thm}

\begin{thm} If $X$ is a \textit{Chow semi-stable} hyperplane arrangement, $\hat{h}(X) > 0$. \end{thm}

\noindent The method of proof is summarized briefly as follows: we first prove that the GIT height of the zero-cycle of $N+1$ distinct points in general linear position in $\Pb^N$ is $0$. Then, using the sub-additivity of the height function and breaking a general  semi-stable zero-cycle into pieces, each of which look like the above zero-cycle (up to a rational coefficient), we get nonnegativity for  semi-stable zero-cycles. We further show that the equality condition for the semi-stable case is not satisfied in the stable case, thereby checking strict positivity for stable zero-cycles. Finally we use the duality between points and hyperplanes along with a careful comparison of metrics to get the result for hyperplane arrangements.
\vspace{0.1in}

\noindent The structure of the paper is as follows. In section \ref{back}, we review some notions and recall the definition of the GIT height function. In section \ref{pts}, we prove positivity for the case of zero-cycles. In section \ref{hyp}, we prove positivity for the case of hyperplane arrangements. 

\subsection*{Acknowledgements}  This work originated in the first author's undergraduate thesis at Princeton advised by Shou-Wu Zhang. We are grateful to him for useful discussions and comments. The second author was supported by the NSF Grant DGE-1148900.

\section{GIT Height Function}\label{back}
\noindent In this section, we review notions of (semi)-stability in geometric invariant theory and define the GIT height function. We begin by fixing some notation. In what follows, let $K$ be a number field and let $M_K$ denote the set of places of $K$. We write the set $M_K = M_{K,f} \sqcup M_{K,\infty}$ where $M_{K,f}$ is the set of non-Archimedean places and $M_{K,\infty}$ is the set of Archimedean places. For any place $v \in M_K$, let $\overline{K_v}$ be the algebraic closure of the completion of $K$ at $v$. There is a unique normalized absolute value $| \cdot |_v$ on $\overline{K_v}$ that extends the $v$-adic absolute value.

\begin{defn}[(Semi)-stability, see \cite{Mum} for more details] For $X$ a scheme and $G$ an algebraic group acting on $X$ whose action is linearized by $L$ a line bundle on $X$, a point $x \in X$ is said to be 
\begin{itemize}
\item \textit{semi-stable} if for some $m \gg 0$, there exists a section $ s \in H^0(X,L^m)$ invariant under the action of $G$ such that $s(x) \ne 0$ and $X_s := \{ p \in X : s(p) \ne 0 \}$ is affine. 

\item \textit{stable} if it is semi-stable, has finite stabilizer in $G$ and the action of $G$ on $X_s$ is closed. 
\end{itemize}
\end{defn}

\noindent Let $V$ be a vector space of dimension $N+1$ over $\Qbar$. Let $X$ be an effective algebraic cycle of pure dimension $n$ in $\Pb(V)$ with degree $d := (X \cdot \Oc(1)^{n})$ and defined over a number field $K$. Using Chow coordinates \cite{DS}, we see that $X$ corresponds to a point $x$ in $Y :=\Pb((\Sym^dV)^{\otimes (N+1)})$. We have an action of $\SL(V)$ via its quotient $\PGL(V)$ on $\Pb(V)$. This is linearized by $\Oc(1)$ on $\Pb(V)$. This group action also induces an action on $\Pb((\Sym^dV)^{\otimes (N+1)})$ which is linearized by $\Oc_Y(1)$. We say that $X$ is  \textit{Chow stable} (resp.\ \textit{Chow semi-stable}) if the corresponding Chow point is stable (resp.\  \textit{semi-stable}). We omit the adjective Chow in what follows.
\vspace{0.1in}

\noindent Note that $Y$ has a homogeneous coordinate $(z_\alpha)$, where $\alpha = (\alpha_{i,j})$ for $0 \leq i \leq n$, $0 \leq j \leq N$, and $\alpha_{i,j} \geq 0$ are multi-indices for monomials of degree $(d, \ldots, d)$ on $\left(\Pb^N\right)^{n+1}$. We now define the Chow metric on $\Oc_Y(1)$.

\begin{defn}[Chow metric] For a point $z$ of $Y$ and $s$ a section of $\Oc_Y(1)$, we define the Chow metric
$$\log ||s||_{\text{Ch}}(z) = \log |s(z)|  - \frac{1}{2}d(n+1)\sum_{j=1}^N \frac{1}{j} - \int_{S(\Cb^{N+1})^{n+1}} \log \left|\sum_\alpha z_\alpha x_\alpha  \right|dx $$ where $x_\alpha = \displaystyle \prod_{i,j}x_{i,j}^{\alpha_{i,j}}$ are monomials on $\Cb^{(N+1)(n+1)}$ and $S(\Cb^{N+1})$ is the unit sphere on $\Cb^{N+1}$.

\end{defn}

\noindent We define the GIT height function of a Chow semistable algebraic cycle $X$ in terms of its archimedean and non-archimedean components. Let $s \in H^0(Y, \Oc_Y(m))$ be any section for any $m > 0$ such that $s(X) \ne 0$. For computations, we often use sections that are invariant under the $\SL(V)$-action. 

\begin{defn}[GIT height, \cite{Z}] The GIT height function is defined as follows: 

$$\hat{h}(X) := \frac{1}{[K:\Qb]} \sum_{v \in M_K} \inf_{g \in \SL(N+1, \overline{K_v})} \left( \frac{-\log||s||_v^m(g \cdot X)}{m} \right),$$ 

where for $v \in M_{K,f}$ we have 
	$$|| s||_v^m(x) := \frac{|s(x)|_v^m}{\displaystyle \max_\alpha |x_\alpha|^m}$$ 

and for $v \in M_{K, \infty}$ we have 
	$$ ||s||_v^m(x) := ||v^*s||_{\text{Ch}}^m(x).$$

\end{defn}

\section{Positivity for zero-cycles}\label{pts}

\noindent In this section, we calculate the GIT height function for certain zero-cycles in $\Pb^N$ defined over $\Qbar$. We then use these calculations to prove the positivity conjecture for semi-stable zero-cycles in $\Pb^N$. Consider a zero-cycle $X = \displaystyle \sum_{i=1}^d P_i$  where $P_i := [P_{i,0} : \ldots : P_{i,N}]$ are (not necessarily distinct) points in $\Pb^N$. We have the Chow morphism $$\varphi : (\Pb(V))^d \to \Pb(\Sym^d(V))$$ defined by sending the ordered tuple of points $(P_1, \ldots, P_d)$ to the corresponding Chow coordinate. It's easy to see that $\varphi^*\Oc_{\Pb(\Sym^d(V))}(1) = \Oc(1, \ldots, 1)$. We equip $(\Pb(V))^d$ with the product of the pull-backs of individual Fubini-Study metrics on $\Pb(V)$ along the $i$-th projection $\pi_i$. The following proposition is a special case of Theorem 1.6 of \cite{Z}. It characterizes the Chow metric for zero-cycles.

\begin{prop}[Theorem 1.6,\cite{Z}]\label{fubinistudy} We have that $$\varphi^*|| \cdot ||_{\text{Ch}} = \prod_{j= 1}^d \pi_j^*|| \cdot ||_{\text{FS}}$$ where $|| \cdot ||_{\text{FS}}$ is the Fubini-Study metric on $\Pb(V)$.\end{prop}

\noindent Since we are interested in computing infima over the various local $\SL_{N+1}$-orbits, we work with a GIT quotient of $\operatorname{Sym}^{d}(\mathbb P^N)$ by the action of $\SL_{N+1}$. Since $\operatorname{Sym}^d (\mathbb P^N)$ is itself a GIT quotient, we will do the GIT in one step. It's well-known that that $\mathbb P^N = \mathbb A^{N+1} / \mathbb G_m$, where $\mathbb G_m$ acts by scaling the coordinates of $\Ab^{N+1}$ and acts on the trivial line bundle by the standard character of $\mathbb G_m$. Taking symmetric powers gives us 

$$\operatorname{Sym}^d (\mathbb P^N) =\left( \mathbb P^N\right)^d/S_d = \mathbb A^{d(N+1)} / \left (\mathbb G_m^d\right) \ltimes S_d.$$ 

\noindent Here the action on the trivial line bundle comes from multiplying the $d$ characters of each of the individual $\mathbb G_m$'s together. The action of $\SL_{N+1}$ on $\Pb^N$ is linearized by $\OO_{\Pb^N}(1)$ and so we can write the GIT quotient of $\Sym^d(\Pb^N)$ by $\SL_{N+1}$ as follows:
$$\operatorname{Sym}^d(\mathbb P^N)/\SL_{N+1} = \mathbb A^{d(N+1)} / \left (\mathbb G_m^d\right) \ltimes S_d \times \SL_{N+1},$$ 

\noindent Explicitly, this space can be described as a quotient of the affine space of $(N+1) \times d$ matrices, where $\SL_{N+1}$ acts by left multiplication, $\mathbb G_m^d$  acts by right multiplication, and $S_d$ acts by permuting the columns. 
We are interested in the ring of invariant sections
$$R_{N,d} := \left(\bigoplus_{m \geq 0} H^0(\Pb^{d(N+1)-1},\OO_{\Pb^{d(N+1)-1}}(m))\right)^{\left(\left (\mathbb G_m^d\right) \ltimes S_d \times \SL_{N+1}\right)}$$ 

\noindent In other words, $R_{N,d}$ comprises precisely those regular functions on the space of $(N+1) \times d$ matrices which are invariant under left multiplication by $\SL_{N+1}$, permutations of the columns by $S_d$, and the right multiplication by diagonal matrices with determinant $1$. We note that the ring of invariants of just the $\SL_{N+1}$ action on $(N+1)\times d$ matrices is generated by the $(N+1)\times (N+1)$ minors. For $d=N+1$, there is just one $(N+1) \times (N+1)$ minor, namely the determinant $\Delta$, which is invariant under the action of $\mathbb G_m^{d}$. The action of $\sigma \in S_d$ by permuting the columns of $(N+1) \times (N+1)$ matrices multiplies $\Delta$ by $\text{sign}(\sigma) \in \{ \pm 1 \}$. It follows that $\Delta^2$ is invariant under all three actions and in fact generates the ring of invariants.

\begin{thm}\label{base case} If $C$ is a zero-cycle of $N+1$ distinct points in $\mathbb P^N$, we have that $\hat{h}(C) = 0$. In particular, the only elements of $\SL_{N+1}(\Cb)$ that minimize any given Archimedean contribution to the height are those that send the $N+1$ column vectors to $N+1$ distinct pairwise orthogonal vectors.  \end{thm}

\begin{proof} By the above argument, we know that $C$ corresponds to an $(N+1) \times (N+1)$ matrix. Since $R_{N,N+1}$ is generated by the square of the determinant, $C$ is semi-stable if and only if the corresponding $(N+1) \times (N+1)$ matrix is invertible. Furthermore since the left action of $\SL_{N+1}$ on the set of $(N+1) \times (N+1)$ invertible matrices is transitive up to scaling by the $\Gb_m$-action, we have a unique semi-stable equivalence class and we can pick an arbitrary representative, say the identity matrix $I_{N+1}$.
\vspace{0.1in}

\noindent We first compute the non-Archimedean portion of the GIT height. By Proposition 1 of \cite{Bu}, for a fixed $v \in M_{K,f}$, we have that minimality in the $\SL_{N+1}(\overline{K_v})$-orbit is the same as residual semistability. Since $I_{N+1}$ has determinant $1$, its reduction modulo $v$ is nonzero and is semistable. Hence, the non-Archimedean portion of the GIT height is \begin{align*} \frac{1}{[K:\Qb]} \sum_{v \in M_{K,f}} \inf_{g \in \SL(N+1, \overline{K_v})} \left( \frac{-\log||\Delta^2||_v^{2(N+1)}(g \cdot I_{n+1})}{2(N+1)} \right)  &= \frac{1}{[K:\Qb]} \sum_{v \in M_{K,f}}  \left( \frac{-\log||\Delta^2||_v^{2(N+1)}(I_{n+1})}{2(N+1)} \right) \\ &= 0 \end{align*}

\noindent Now, we fix $v \in M_{K,\infty}$ and compute the $v$-component of the height. By Proposition \ref{fubinistudy}, this is equal to $$ \inf_{A = (a_{ij}) \in \SL_{N+1}(\Cb)} \frac{\displaystyle \frac{1}{2}\sum_{i=1}^{d} \log \left(\sum_{j=1}^{N+1} |a_{ij}|^2 \right)}{d}$$ which is nonnegative by Hadamard's inequality. If $A$ is the identity, we get $0$, so the infimum is $0$. Summing up over all $v \in M_{K,\infty}$, it follows that the non-Archimedean portion of the GIT height of $C$ is $0$ and hence $\hat{h}(C) = 0$.
\vspace{0.1in}

\noindent Because Hadamard's inequality is sharp if and only if the columns of the matrix are orthogonal, only these matrices minimize the local contribution.
\end{proof}

\noindent The following proposition checks that the GIT height is sub-additive. 
\begin{prop}\label{additivity} For $C_1$ and $C_2$ semi-stable algebraic cycles of dimension $n$ on $\mathbb P^N$, we have \[\hat{h}(C_1+ C_2) \geq \hat{h}(C_1) + \hat{h}(C_2) \]

\noindent Moreover, the inequality is sharp only if, at each place, each element of $\SL_{N+1}$ that minimizes the local contribution to the height of $C_1+C_2$ also minimizes the local contributions of both $C_1$ and $C_2$. \end{prop}

\begin{proof} When we add two algebraic cycles, their Chow points are multiplied using the multiplication map $$\Sym^d V \otimes \Sym^e V \to \Sym^{d+e} V,$$ hence $$ (\Sym^d V) ^{ \otimes n+1} \otimes (\Sym^e  V^{ \otimes n+1}) \to (\Sym^{d+e} V)^{\otimes n+1}.$$ So it is clear that $\displaystyle \sum_\alpha z_\alpha x_\alpha$ multiplies when we add two cycles, hence $\displaystyle \log \left | \sum_\alpha z_\alpha x_\alpha\right |$ is additive. 

\noindent To ensure that $\displaystyle \log |s(z)|$ is additive, consider the map $$f : \Pb(\Sym^d V ^{ \otimes n+1}) \times \Pb(\Sym^e V ^{ \otimes n+1}) \to \Pb(\Sym^{d+e} V ^{ \otimes n+1}).$$ Since this is linear, we have that $f^*\Oc(1) = \Oc(1,1)$. If we have a non-vanishing section $s(C_1 + C_2) \ne 0$, we choose the corresponding sections for $C_1$ (respectively $C_2$) by pulling back along $f$ and then along $$f' : \Pb(\Sym^d V ^{ \otimes n+1}) \times \{ C_2 \} \to \Pb(\Sym^{d+e} V ^{ \otimes n+1}).$$ Hence $$\log ||s||_{\text{Ch}}(C_1 + C_2) = \log ||s||_{\text{Ch}}(C_1 )+ \log ||s||_{\text{Ch}}( C_2)$$
Finally, observe that since $$\inf_{a \in A, b \in B}( a+ b) \geq (\inf_{a \in A} a) +   (\inf_{b \in B} b),$$ so $$ \inf_{g \in \SL(N+1, \overline{K_v})} \left( \frac{-\log||s||_v^m(g \cdot C_1 + g \cdot C_2 )}{m} \right)$$ $$\geq  \inf_{g_1 \in \SL(N+1, \overline{K_v})} \left( \frac{-\log||s||_v^m(g_1 \cdot C_1)}{m} \right) +  \inf_{g_2 \in \SL(N+1, \overline{K_v})} \left( \frac{-\log||s||_v^m(g_2 \cdot C_2)}{m} \right)$$ and by summing over $v$, we have the desired inequality. Clearly this inequality is only sharp if the minimizing values of $g$ for $C_1$ and $C_2$ are the same.

\end{proof}

\noindent The following proposition tells us precisely when a degree $d$ zero-cycle is (semi-)stable. 

\begin{prop}[Proposition 3.4, \cite{Mum}]\label{H-M} If $C$ is a degree $d$ zero-cycle on $\mathbb P^N$, it is semi-stable if and only if for any $k > 0$, no more than $dk/(N+1)$ of the points lie in any $(k-1)$-dimensional subspace of $\mathbb P^N$. In terms of column vectors, $C$ is semi-stable if and only if no more than $dk/(N+1)$ of the vectors lie in any $k$-dimensional subspace of $\mathbb C^{N+1}$. $C$ is stable if and only if all the inequalities are strict. \end{prop}

\noindent Putting things together, we get the following theorem for semi-stable zero-cycles. 

\begin{thm} If $C$ is a semistable zero-cycle of degree $d$ on $\mathbb P^N$, then $\hat{h}(C) \geq 0$. \end{thm}

\begin{proof}  We note that for any $C' = \sum_{i=1}^n a_i C_i'$ where $a_i \in \Qb$, clearing denominators and repeatedly using Proposition \ref{additivity} yields $$\hat{h}(C') \geq \sum_{i=1}^n a_i \hat{h}(C_i').$$ Thus it will be sufficient to write $C$ as a nonnegative rational linear combination of the algebraic cycle consisting of $N+1$ distinct points in general linear position, which is a semistable cycle on $\mathbb P^N$, by Proposition \ref{H-M}. \vspace{0.1in}

\noindent We view $C$ as a \textit{set-with-multiplicity} $S$ of $d$ vectors in $\mathbb C^n$. Since we are working with rational linear combinations, it will be convenient to let the multiplicities be nonnegative rational numbers. Whenever we count vectors, unless we specify otherwise, we will count them with multiplicity. A set of vectors with nonnegative rational multiplicity is said to be a \textit{configuration of vectors} and a configuration of $d$ vectors in $\mathbb C^{m}$ such that no more than $ dk/m$ vectors lise in any $k$-dimensional subspace is said to be \textit{semistable configuration}. A configuration that is a nonnegative rational combination of some number of bases of $\mathbb C^{N+1}$ is said to be a \textit{decomposable configuration}.
\vspace{0.1in}

\noindent Let $S$ be a semistable configuration in $\mathbb C^m$. We check that every semistable configuration is decomposable by inducting on the number $\ell$ of distinct vectors. If $\ell = 0$, there is nothing to prove. Otherwise, we split into two cases, depending on whether there is a nontrivial proper subspace of dimension $k$ that contains exactly $kd/m$ vectors. 
\vspace{0.1in}

\noindent Suppose there is a nontrivial proper subspace $V \subset \mathbb C^m$ of dimension $k$ that contains exactly $kd/m$ vectors of $S$. The configuration of vectors in $S$ contained in $V$ (with the same multiplicities) is a semistable configuration in $\mathbb C^m$ by Proposition \ref{H-M}, and contains fewer distinct vectors than $S$ and is therefore, by induction, decomposable. The configuration of vectors in $S$ not contained in $V$ (with the same multiplicities), projected to $\mathbb C^{m}/V$ is  semistable by Proposition \ref{H-M}, and also contains fewer distinct vectors and is hence decomposable. The union of any basis of $S \cap V$ with any subset of $S \text{ } \backslash \text{ }V$ that projects to a basis of $\mathbb C^{m}/V$ is a basis of $\mathbb C^{m}$. So we can combine the two decompositions to obtain a decomposition of $S$. 
\vspace{0.1in}

\noindent Suppose now that no nontrivial proper subspace of dimension $k$ contains exactly $kd/m$ vectors of $S$. Choose any subset of $S$ that is a basis. This is possible unless all of $S$ is contained in a single proper subspace, which would violate semistability. Subtract a multiple of this basis from $S$ - the largest multiple such that none of the semistability conditions is violated and the multiplicity of each of the vectors remains nonnegative. Since all the inequalities are linear with rational coefficients, this largest multiple must be a rational number. At this point, we still have a semistable configuration of vectors, but either one subspace of dimension $d$ now contains exactly  $kd/m$ vectors, and we are in the first case, or the multiplicity of one of the vectors becomes zero, which means the number of distinct vectors is reduced by one, so the new configuration is decomposable by induction, and we can add the basis back in to get a decomposition of $S$.

\end{proof}

\noindent Finally, we show that for stable cycles, we have positivity and not merely nonnegativity.

\begin{thm} If $C$ is a stable zero-cycle of degree $d$ on $\mathbb P^N$, then $\hat{h}(C) > 0$. \end{thm}

\begin{proof} Choose an element $A$ of $\SL_{N+1}(\mathbb C)$ that minimizes the local contribution to the height of $C$ at $v \in M_{K,\infty}$. We will write $C$ as a rational linear combination of two semistable zero-cycles. By Proposition \ref{additivity}, this implies that $\hat{h}(C) \geq 0$. If we can show that $A$ does not minimize the local contribution of one of these two semistable zero-cycles, then Proposition \ref{additivity} tells us that the inequality is not sharp and hence $\hat{h}(C)>0$. We do this by choosing one of our zero-cycles $C'$ to correspond to a configuration of $N+1$ vectors that is independent, but (after applying $A$) is not pairwise orthogonal. Thus $C'$ is semi-stable but by Theorem \ref{base case} its local contribution is not minimized by $A$.
\vspace{0.1in}

\noindent Assume that we have $N+1$ distinct points $\{P_1, \ldots, P_{N+1}\}$ corresponding to vectors in $C$ that are independent but not pairwise orthogonal. Then we claim that $$C - \frac{P_1 + \dots + P_{N+1}} {(N+1)^2}$$
is semistable. We can see this by checking the relevant inequalities from Proposition \ref{H-M}. Since both sides of these inequalities are rational numbers of denominator dividing $N+1$, they are strict by at least $\displaystyle \frac{1}{N+1}$, and subtracting a cycle of degree at most $\displaystyle \frac{1}{N+1}$ preserves semistability. Thus we will have succesfully decomposed $C$ into two cycles, one of which  does not have its height minimized.
\vspace{0.1in}

\noindent So it remains to show that given a stable configuration of $d$ vectors in $\mathbb C^{N+1}$, we may find $N+1$ that are linearly independent but not pairwise orthogonal. To do this, pick any vector $v_1$ in $C$. Then pick any vector $v_2$ that is not a multiple of $v_1$ but is also not orthogonal to $v_2$. This is possible because if every vector in $C$ is either a linear multiple of $v_1$ or orthogonal to $v_1$, then those two subspaces of dimensions $1$ and $N$ would contain all vectors, but by stability one contains strictly less than $1/(N+1)$ of the vectors and the other contains strictly less than $N/(N+1)$ of the vectors. Finally, inductively pick $N-1$ additional vectors that are each independent of the previous ones. This is possible because otherwise all the vectors would lie in a proper subspace, which is impossible by semistability.

 \end{proof}

\section{Positivity for hyperplane arrangements}\label{hyp}

\noindent In this section, we prove the positivity of the GIT height function for semi-stable hyperplane arrangements. Recall that a \textit{hyperplane arrangement} is an algebraic cycle comprised of a finite set-with-multiplicity of hyperplanes in $\Pb(V)$. Since $V \cong (V^\vee)^\vee$, a vector in $V$ corresponds to a linear functional on $V^\vee$ whose zero-set defines a hyperplane in $V^\vee$. Projectivizing, we see that a point in $\Pb(V)$ corresponds to a hyperplane in $\Pb(V^\vee)$ and thus we have a map from the Chow variety of zero-cycles in $\mathbb P(V)$ to the Chow variety of $(N-1)$-cycles in $\mathbb P(V^{\vee})$ that sends an arrangement of points to the dual hyperplane arrangement. Call this map $\varphi$.

\vspace{0.1in}

\noindent We will fix a generator of $\bigwedge^{N+1} V$ and hence an isomorphism $\bigwedge^{N} V \cong V^{\vee}$ and $\bigwedge^{N} V^{\vee} \cong V$. This will not affect the invariance of any construction under $SL_{N+1}$.

\begin{lem}\label{chow formula} We may view the Chow coordinates of an $n$-cycle of degree $d$ in $\mathbb P(V)$ as a degree $d$ multilinear function on $n$-tuples of elements of $V^{\vee}$. With this point of view, the Chow coordinates of $\varphi(X)$ are just the Chow coordinates of $X$ composed with the map $x_1 \wedge \dots \wedge x_N$ from $V^n$ to $V^{\vee}$. \end{lem} 

\begin{proof}The Chow coordinate of a hyperplane $\varphi(v)$ in $V^{\vee}$ that corresponds to an element $v \in V$ is simply a function $$F_v : V^{\otimes N} \to \bigwedge^{N+1}V \cong k$$ such that $$F_v(x_1\otimes \ldots \otimes x_N) = v \wedge x_1 \ldots \wedge x_N.$$ 

\noindent From this definition we can see that the Chow coordinates of a hyperplane arrangement is a product of functions of this type, which we can write as $$x_1 \otimes \ldots \otimes x_N \mapsto  \prod_{i=1}^d (v_i \wedge x_1 \dots \wedge x_N)$$ which is the composition of the function $$x \mapsto \prod_{i=1}^d v_i \wedge x$$ with the function $$x_1 \otimes \ldots \otimes x_N \mapsto x_1 \wedge \dots \wedge x_N.$$ The Chow coordinate of the arrangement of points is just $x \displaystyle \mapsto \prod_{i=1}^d v_i \wedge x$.    Hence the hyperplane is indeed such a composition. \end{proof}

\noindent The following proposition summarizes some elementary properties of $\varphi$.

\begin{prop}\label{props} $\varphi$ is $\SL(V)$-equivariant and $$\varphi^*\Oc_{\Pb[(\Sym^d(V^\vee))^{\otimes N}]}(1) = \Oc_{\Pb(\Sym^d(V)}(1).$$ Furthermore $\varphi$ preserves semi-stability. \end{prop}

\begin{proof} It's clear that $\varphi$ is $\SL(V)$-equivariant. To compute the pull back of $\Oc(1)$ along $\varphi$, we use Lemma \ref{chow formula} to describe $\varphi$ in Chow coordinates. Because composition with a function is linear, the Chow coordinate of the hyperplane arrangement can be computed from the Chow coordinate of the zero-cycle by a linear map $\Pb(\Sym^d(V)) \to \Pb[(\Sym^d(V^\vee))^{\otimes N}]$. Because the map is linear, we see that $$\varphi^*\Oc_{\Pb[(\Sym^d(V^\vee))^{\otimes N}]}(1) = \Oc_{\Pb(\Sym^d(V))}(1).$$ This fact, plus the equivariance, implies that $\varphi$ preserves semi-stability. 
\end{proof}

\noindent The following proves the strict positivity for semistable hyperplane arrangements.

\begin{thm} If $X$ is a semi-stable hyerplane arrangement on $\mathbb P^N$ and $N > 1$, $$\hat{h}(\varphi(X)) > \hat{h}(X).$$ \end{thm}
\begin{proof} To compare the GIT heights we need a section $s$ of $\Oc_{\Pb(\Sym^d(V)}(1)$ and a section $s'$ of $\Oc_{\Sym^d(V^\vee))^{\otimes N}}(1)$. We need these sections to be comparable so we can compare the height at each place. To this end, we view $\Sym^d(V^\vee))^{\otimes N}$ as the space of functions of multidegree $d$ on $N$ copies of $V$. Then choose $s'$ to be evaluation of that function on $N$ fixed vectors $v_1, \dots, v_N$ of $V$. 
\vspace{0.1in}

\noindent Then view $\Sym^d(V)$ as functions of degree $d$ on $V^{\vee}$, and let $s$ be evaluation of the function on $v_1 \wedge \dots \wedge v_N \in \wedge^{n-1} V  \cong V^\vee$. Now, because of Lemma \ref{chow formula}, $$s(X) = s'(\varphi(X)).$$

\noindent We are going to  show that for $X$ a zero-cycle of degree $d$ and $v \in M_{K,f}$, we have for $v\in M_{K,f}$ $$ ||s||'_v(\varphi(X))=||s||_v(X)$$ and for $v\in M_{K,\infty}$, we have   $$ ||s||'_v(\varphi(X))=||s||_v(X)-C$$ for an explicit positive constant $C$. It will follow that $$\hat{h}(\varphi(X))=\hat{h}(X)+ C.$$

\noindent First let's verify the non-archimedean case. Since we already saw  $s(X) = s'(\varphi(X))$, it is sufficient to show that the maximum of the $v$-adic valuations of the Chow coordinates of $X$ and $\varphi(X)$ are the same. The Chow coordinates of $\varphi(X)$ come from the Chow coordinates of $X$ by a linear map with integer coefficients, so $\varphi$ cannot increase the maximum $v$-adic valuation of the entries, and it only diminishes them if it sends an integer vector that is not a multiple of $p$ to a vector that is a multiple of $p$, which means it fails to be injective modulo $p$. Since the linear map is composition with a certain fixed function, that is the same as showing that the function being composed with is surjective modulo $p$. But that is the map $V \times \dots \times V \to \wedge^N V$, which is surjective by definition.
\vspace{0.1in}

\noindent Next let's verify the infinite case. The key term is the integral, which we will handle first. Recall again that the Chow coordinates of $\varphi(X)$ are defined by composition with the wedge map, so the integral $$I = \int_{S(\mathbb C^{N+1})^{N+1}} \log | \sum_\alpha z_\alpha x_\alpha | dx= \int_{S(\mathbb C^{N+1})^{N}} \log | f(v_1 \wedge \dots \wedge v_N) | dx.$$ For $v_1, \dots, v_N$ on the unit sphere, we can write $v_1 \wedge \dots \wedge v_N$ as a point $a(v_1,\dots, v_N)$ on the unit sphere times a nonnegative radius $|v_1 \wedge \dots \wedge v_N|$. Now using the fact that $f$ is homogeneous of degree $d$, we compute

\begin{align*}
 I &=\displaystyle \int_{S(\mathbb C^{N+1})^{N}} \log \left| f(a (v_1,\dots, v_N) |v_1 \wedge \dots \wedge v_N|) \right| dx \\ &=\int_{S(\mathbb C^{N+1})^{N}} \log \left| f(a (v_1,\dots, v_N) )|v_1 \wedge \dots \wedge v_N|^d \right| dx \\ &=\int_{S(\mathbb C^{N+1})^{N}} \left(\log | f(a (v_1,\dots, v_N) ) | +d \log |v_1 \wedge \dots \wedge v_N|\right)dx \\ &=\int_{S(\mathbb C^{N+1})^{N}} \log | f(a (v_1,\dots, v_N) ) | dx + d\int_{S(\mathbb C^{N+1})^{N}} \log |v_1 \wedge \dots \wedge v_N|dx \end{align*}

\noindent Now for $v_1, \dots, v_N$ evenly distributed on the unit sphere, $a(v_1,\dots,v_n)$ is also evenly distributed on the unit sphere. Hence the pushforward of the uniform measure on $S(\mathbb C^{N+1})^{N+1} $ to $S(\mathbb C^{N+1})$ along the map $a$ is the uniform measure on $S(\mathbb C^{N+1})$ and so $$ \int_{S(\mathbb C^{N+1})^{N+1}} \log | f(a (v_1,\dots, v_N) ) | dx = \int_{S(\mathbb C^{N+1})} \log |f(x)| dx,$$ which is exactly the integral term in the formula for the Chow metric of $X$.
\vspace{0.1in}

\noindent By matching terms we get, because $n=N-1$ for hyperplanes and $n=0$ for points, $$\log ||s'||_{Ch} (\varphi(X))- \log ||s||_{Ch}(X) = - \frac{1}{2} d (N-1) \sum_{j=1}^N \frac{1}{j} -   d\int_{S(\mathbb C^{N+1})^{N+1}} \log |v_1 \wedge \dots \wedge v_N|dx $$ so by letting $$C' = \frac{1}{2} (N-1) \sum_{j=1}^N \frac{1}{j} + \int_{S(\mathbb C^{N+1})^{N}} \log | v_1 \wedge \dots \wedge v_N|dx $$ we get what was claimed except we must check that $C'>0$.
\vspace{0.1in}

\noindent To check this, observe that $\log |v_1 \wedge \dots \wedge v_N|$ is the sum of $i$ from $1$ to $N$ of the log of the length of $v_i$ projected onto the perpendicular subpaces of $(v_1,\dots, v_{i-1})$. Because $|v_1|=1$, the first term vanishes. Clearly the integral over all $v_i$ of that logarithm does not depend on the particular subspace. So we can take the subspace to be generated by the first $i$ basis vectors and we have

\begin{align*} \int_{S(\mathbb C^{N+1})^{N}} \log | v_1 \wedge \dots \wedge v_N|dx &= \sum_{i=2}^N  \int_{(x_0,\dots,x_N) \in S(\mathbb C^{N+1})}\frac{1}{2}  \log (\sum_{j=i-1}^N |x_j|^2) dx \\ &> \sum_{i=2}^N  \int_{(x_0,\dots,x_N) \in S(\mathbb C^{N+1})} \frac{1}{2} \log  |x_N|^2  dx \\ &=- \sum_{i=2}^N \frac{1}{2} \sum_{j=1}^N \frac{1}{j} = -\frac{1}{2} (N-1) \sum_{j=1}^N \frac{1}{j} \end{align*}The last step above is a computation of Stoll from the proof of Theorem 1.6 in \cite{Z}. It's now clear that $C' > 0$ and we're done.
\end{proof}
\bibliographystyle{plain}
\bibliography{biblio}
\end{document}